\documentclass[10pt]{amsart}
\usepackage{amssymb,latexsym,amsmath}
\usepackage{graphics}                 
\usepackage{color}                    
\usepackage{hyperref}                 
\usepackage[all]{xy}

\begin{document}

\newtheorem{definition}{Definition}
\newtheorem{theorem}{Theorem}[section]
\newtheorem{proposition}[theorem]{Proposition}
\newtheorem{lemma}[theorem]{Lemma}
\newtheorem{corollary}[theorem]{Corollary}
\newtheorem{question}[theorem]{Question}
\newtheorem{remark}[theorem]{Remark}
\newtheorem{example}[theorem]{Example}
\newtheorem{conjecture}{Conjecture}

\newtheorem{correction}{Correction}

\def\A{{\mathbb{A}}}
\def\C{{\mathbb{C}}}
\def\F{{\mathbb{F}}}
\def\G{{\mathbb{G}}}

\def\L{{\mathbb{L}}}
\def\P{{\mathbb{P}}}
\def\Q{{\mathbb{Q}}}
\def\Z{{\mathbb{Z}}}
\def\Ch{{\rm Ch}}
\def\p{{\mathbf{p}}}

\def\char{{\rm char}}
\def\mod{{\rm mod}}
\def\sp{{\rm SP}}
\def\supp{{\rm supp}}
\def\rank{{\rm rank}}
\def\Spec{{\rm Spec}}

\def\mC{{\mathcal{C}}}
\def\cZ{{\mathcal{Z}}}
\def\D{{\mathcal{D}}}

\title{On Additive invariants of  actions of additive and multiplicative groups}
\author{Wenchuan Hu}
\date{October 25, 2010}


\thanks{This material was based upon work supported by
the NSF under agreement No. DMS-0635607.}

\address{
School of Mathematics\\
Sichuan University\\
Chengdu 610064\\
P. R. China
}

\email{huwenchuan@gmail.com}

\begin{abstract}

The additive invariants  of an algebraic variety is calculated in terms of those of the fixed point set under the action
of additive and multiplicative groups, by using  Bia{\l}ynicki-Birula's fixed point formula for a projective algebraic
set with a $\G_m$-action or $\G_a$-action.

The method is also generalized to calculate certain additive invariants for Chow varieties. As applications, we obtain the Hodge
polynomial of Chow varieties in characteristic zero and  the number of points  for Chow varieties over finite fields.

As applications, we obtain the $l$-adic Euler-Poincar\'{e} characteristic for the Chow varieties of certain projective varieties over an
algebraically closed field of arbitrary characteristic. Moreover, we show that the virtual Hodge $(p,0)$ and $(0,q)$-numbers of the
Chow varieties  and affine group varieties are zero for all $p,q$ positive.

\end{abstract}

\maketitle
\pagestyle{myheadings}
 \markright{On Additive invariants}

\tableofcontents

\section{Introduction}
In this paper we generalize a method of Bia{\l}ynicki-Birula (cf. \cite{Bialynicki-Birula2}) in studying the fixed point schemes
under actions of additive and multiplication group schemes and apply it to calculate additive invariants of projective varieties admitting
one of these actions, especially to  affine group varieties and Chow varieties.

Recall that  an additive invariant $\lambda$ on the category $Var_{K}$
of algebraic varieties (a variety means a reduced and irreducible scheme) over a  field $K$ with values in a ring $R$,
is a map
$$
\lambda:Var_{K}\to R
$$
such that
$$\left\{
\begin{array}{lll}
 \lambda(X)=\lambda(X')&\hbox{for $X\cong X'$,}\\
 \lambda(X)=\lambda(Y)+\lambda(X-Y)& \hbox{for $Y$ closed in  $X$,}\\
 \lambda(X\times Y)=\lambda(X)\cdot \lambda(Y)&\hbox{for every $X$ and  $Y$.}
\end{array}\right.
$$

Examples of additive invariants includes the Euler characteristic, the $l$-adic Euler-Poincar\'{e} characteristic,
the Hodge polynomial, counting points, etc. For more examples and details on additive invariants,
the reader is referred to Loeser's lecture \cite{Loeser}.

Our motivation  comes from  the computation of the Euler characteristic of the Chow variety of complex projective spaces by Lawson and
 Yau (cf. \cite{Lawson-Yau}). More precisely, it is from the calculation of the Euler characteristic of
 the complex Chow variety $C_{p,d}(\P^n_{\C})$ (or simply $C_{p,d}(\P^n)$ if there is no confusion) parameterizing effective $p$-cycles of
degree $d$ in the complex projective space $\P^n$. The following formula was shown to hold:

\begin{theorem}[The Lawson-Yau formula] \label{Th01.1} For all $n,p,d\geq 0$, one has
$$
\chi(C_{p,d}(\P^n))=\big(^{v_{p,n}+d-1}_{\quad\quad d}\big),
$$
where $v_{p,n}=(^{n+1}_{p+1})$ and $\chi(M)$ is the Euler  characteristic of $M$.
\end{theorem}

 Lawson and Yau  use a fixed point formula of a weakly holomorphic $S^1$-action in their computation for the Euler
 characteristic of Chow varieties.
 We observe that it would work nicely for other interesting additive invariants once we have corresponding fixed point
 formulas. The basic tool we will use in our proof is a mild generalized version of the following  fixed point formula for $l$-adic
 Euler-Poincar\'{e} characteristic, as proved by Bialynicki-Birula.

Let $X$ be a  projective algebraic subset over a field $K$ with a $\G_m$-action. Note that
$\G_m\cong \Spec K[t,t^{-1}]$.
 That is, there is a morphism
$\phi:\G_m\times X\to X$ such that
$\phi(1,x):=x$  and $\phi(t_1 t_2,x)=\phi(t_1,\phi(t_2, x))$.

\begin{theorem}[Bia{\l}ynicki-Birula, \cite{Bialynicki-Birula2}]\label{Th01.2}
Let $X$ be a  projective algebraic subset over an algebraically closed field $K$ with a $\G_m$-action. Then
$$\chi(X,l)=\chi(F,l),$$
where $F$ is the fixed point set of this action.
\end{theorem}

The first main result in this paper is the following statement.
\begin{theorem}[Corollary \ref{Cor2.4}]\label{Th01.4}
Let $\lambda: Var_K\to R$ be an additive invariant  satisfying $\lambda(\G_m)=0$ (resp. $\lambda(\G_a)=0$). Then
$\lambda(X)=\lambda(X^{\G_m})$ (resp. $\lambda(X)=\lambda(X^{\G_a})$).
\end{theorem}

Note that $K$ is not required to be algebraically closed in this theorem.

By applying this to the Chow variety $C_{p,d}(\P^n)_K$ parametrizing  effective $p$-cycles of degree $d$ in the projective
space $\P^n_K$ over an algebraically closed field $K$, we obtain the following result.
\begin{theorem}\label{Th1.4}
Let $\lambda: Var_K\to R$ be an additive invariant  satisfying $\lambda(\G_m)=0$ and $\lambda(\Spec K)=1$. Then
$$
\lambda(C_{p,d}(\P^n)_K)=\big(^{v_{p,n}+d-1}_{\quad\quad d}\big)\in R.
$$
\end{theorem}

In particular, if $\lambda(-)=\chi(-,l)$ is the $l$-adic Euler-Poincar\'{e} characteristic  and $R=\Z$, then we obtain the
 $l$-adic Euler-Poincar\'{e} characteristic  for $C_{p,d}(\P^n)_K$.

\begin{corollary}\label{cor1.5}
\begin{equation}
\chi(C_{p,d}(\P^n)_K,l)=\big(^{v_{p,n}+d-1}_{\quad\quad d}\big).
\end{equation}
\end{corollary}

In particular, we obtain the following results on the virtual Hodge numbers and the virtual Betti numbers of the
Chow variety $C_{m,d}(\P^n_K)$ parameterizing  algebraic $m$-cycles of degree $d$ in $\P^n_K$.
\begin{theorem} Let $K$ be an algebraically closed subfield of $\C$. For integers $n\geq m\geq 0$ and $d\geq 0$,
the  virtual Hodge $(p,0)$ and $(0,q)$-numbers
of the Chow variety $C_{m,d}(\P^n_K)$ are zero for all integers $p, q>0$. Moreover, the virtual Hodge $(p,q)$-numbers
 $\tilde{h}^{p,q}(C_{m,d}(\P^n_K))$ of $C_{m,d}(\P^n_K)$ satisfies the following equation:
$$
\sum_{p-q=i}\tilde{h}^{p,q}(C_{m,d}(\P^n_K))=0
$$
for all $i\neq 0$ and $$
\sum_{p\geq0}\tilde{h}^{p,p}(C_{m,d}(\P^n_K))=\chi(C_{m,d}(\P^n_K)).
$$
\end{theorem}

The method in proving Theorem \ref{Th1.4} is applied to obtain additive invariants for the Chow varieties of general
 toric varieties as well as the Chow varieties parametrizing irreducible cycles in the  product of arbitrary many projective spaces.

\emph{Acknowledgement}. I would like to thank professor V. Voevodsky for  his interesting and
helpful comments on an earlier version of the paper.


\section{A generalization of Bia{\l}ynicki-Birula's method to additive invariants}

Let $A$ be a fixed algebraic variety  over a field $K$ of arbitrary characteristic.
 An algebraic scheme $Y$ is said to be simply equivalent to an algebraic
scheme $X$ if $Y$ is isomorphic to a closed subscheme $X'$ of $X$ and there is  an isomorphism $f:X-X'\to Z\times A$ for some
algebraic scheme $Z$. The smallest equivalence relation containing the relation of  simple $A$-equivalence is called the
$A$-equivalence.

\begin{lemma}
Let $X$, $Y$ and $A$ be algebraic varieties over $K$ and let $\lambda:Var_{K}\to R$ be an additive invariant. Suppose that $X$
is $A$-equivalent to $Y$. If $\lambda(A)=0\in R$, then $\lambda(X)=\lambda(Y)\in R$.

\end{lemma}
\begin{proof}
It is enough to consider the case that $X$ is simply $A$-equivalent to $Y$ since $\lambda$ is an additive invariant.
By definition, there is an open quasi-projective
scheme $U$ of $X$ such that $X-U$ is isomorphic to $Y$ and an isomorphism $f:U\to U'\times A$, where $U'$ is an algebraic scheme.
 In this case, we have $ \lambda(X)= \lambda(Y)+ \lambda(U)= \lambda(Y)+ \lambda(U') \lambda(A)=\lambda(Y)$.
This completes the proof of the lemma.
\end{proof}

\begin{theorem}\label{Th2.2}
Let $G=\G_m$  and suppose that $G$ acts on a reduced and  irreducible algebraic scheme $X$. Then $X$ is $A$-equivalent to
$X^G$, where $X^G$ denotes the fixed point set of the $G$-action and $A=\Spec (K[x,x^{-1}])$.
Likewise, if $X$ admits the action of  the additive group $G=\G_a$  and $A=\Spec (K[x])$, then $X$ is $A$-equivalent to
$X^G$.
\end{theorem}

\begin{proof}In the following $G$ is either $\G_a$ or $\G_m$.
The case is clear if  $X= X^G$. Suppose that $X\neq X^G$.  We will show that  there exists a $G$-invariant  non-empty open
subscheme $U$ of  $X$ isomorphic to $Z\times A$, for some scheme $Z$. To see this, let $U'$ be a non-empty open
irreducible subscheme of $X$ such that  the quotient  $\phi:U'\to U'/G$ exists. The te generic fiber $F$ of $\phi$ is an
algebraic scheme over the field of rational functions $K(U'/G)=K(U')^G$. Moreover,  the fiber $F$ with the action of $G$
is homogeneous. Hence there exists a $K(U'/G)$-rational point in $F$ and $F$ is isomorphic to $G/H$ for some algebraic group
subscheme over $K(U')^G$ of $G$, where the action of $G$ on $G/H$ is induced by translations. By our assumption, the group
scheme $G/H$ is isomorphic to $G$. Hence $X$is birationally $G$-equivalent to some product $U_1'\times G$. Therefore,
$X$ contains an open $G$-invariant subscheme  $U$ which is isomorphic to non-empty open subscheme of $U_1'\times G$.
 Note that a $G$-invariant open subscheme of $U_1'\times G$ is of the form $U_1\times G$ for some open subscheme $U_1$ of $U_1'$. Thus $X$
is $A$-equivalent to $X-U$.  If $X-U=(X-U)^G$, then $X^G=X-U$ and so $X$ is $A$-equivalent to $X^G$.
Otherwise, $X-U\neq (X-U)^G$ then we repeat the above step where $X$ is replaced by $X-U$. Since $X$ is noetherian,
 we obtain a closed subscheme $X_0$ of $X$ such that $X_0=X_0^G$ and $X_0$ is $A$-equivalent to $X$.
 From the construction of $X_0$, we see that $X_0^G=X^G$.
Therefore, $X^G$ is $A$-equivalent to $X$.
\end{proof}

\begin{remark}
The proof above follows from  Bia{\l}ynicki-Birula's argument, where the base field he considered is an algebraically closed field.
However, the proof works for an arbitrary field.
\end{remark}

\begin{corollary}\label{Cor2.4}
Let $\lambda: Var_K\to R$ be an additive invariant  satisfying $\lambda(\G_m)=0$ (resp. $\lambda(\G_a)=0$). Then
$\lambda(X)=\lambda(X^{\G_m})$ (resp. $\lambda(X)=\lambda(X^{\G_a})$).
\end{corollary}

For additive invariants $\lambda$ defined on $Var_K$ to be interesting,
 we require that $\lambda(\Spec(K))=1$. In fact, it follows from
the definition of additive invariants that $\lambda(\Spec(K))$ is either $0$ or $1$. Moreover, if $\lambda(\Spec(K))=0$, then
it follows from the definition that $\lambda\equiv 0$. So  we only consider  non-trivial  additive invariants $\lambda$ below, i.e.,
 $\lambda(\Spec(K))=1$.

\begin{example} Let $\lambda:Var_K\to \Z$ be an additive invariant such that  $\lambda(\G_m)=0$.
Then for  $X=X(\Delta)$ a (possible singular) toric variety associated to a fan $\Delta$, we have $\lambda(X)=d_n(\Delta)$, where
$d_n(\Delta)$ is the number of $n$-dimensional cones in $\Delta$ and $n$ is the dimension of $X$.
\end{example}
\begin{proof}
 For $X=X(\Delta)$ an arbitrary toric variety, we write $X$ as the disjoint union of its orbits
$O_{\tau}$ under $\G_m^{\times n}$. Each orbit $O_{\tau}$ is isomorphic to $(\G_m)^{\times i}$. By assumption, $\lambda(\G_m)=0$.
This implies that
$\lambda((\G_m)^{\times i})=0$ for $i>0$. The number of $0$-dimensional orbits is exactly the number of   $n$-dimensional cones in
$\Delta$, i.e., $d_n(\Delta)$.
\end{proof}

Note that the Euler characteristic of $X(\Delta)$ is also $d_n(\Delta)$ (cf. \cite[Ch.3]{Fulton2}). This is not surprising
since the Euler characteristic $\chi$ is an additive invariant satisfying $\chi(\G_m)=0$ (see the next section).

If a variety $X$ admits a $\G_m$-action with isolated fixed points, then for any additive invariant $\lambda:Var_K\to \Z$
with $\lambda(\G_m)=0$,  $\lambda(X)$ coincides with the cardinality of the fixed point set. In particular, the fixed point set of
$\G_m$-action on an algebraic torus can not have isolated fixed points.

\medskip
A variety is called \textbf{cellular} if there is a filtration
$\emptyset=Y_{-1}\subset Y_0\subset Y_1\subset\cdots\subset Y_N=Y$
such that $Y_i-Y_{i-1}$ is isomorphic to $\C^{  \mu_i} $ for all $i$ (where $0=  \mu_0\leq  \mu_1\leq \cdots$).
\begin{example}
Let $\lambda:Var_K\to \Z$ be an additive invariant such that $\lambda(\G_m)=0$. Then  for
a cellular variety $Y$ as above, one has $\lambda(Y)=N$.
\end{example}

\begin{example}(cf. \cite[Cor.5]{Bialynicki-Birula2})
 Let $\lambda:Var_K\to \Z$ be an additive invariant such that $\lambda(\G_m)=0$. For an algebraic
connected reduced affine group scheme $G$, one has $\lambda(G)=0$ or $1$. Moreover, $\lambda(G)=1$ if and only if
$G$ is unipotent.
\end{example}
\begin{proof}
If $G$ is not unipotent , then it contains a subgroup isomorphic to $\G_m$. The action of $\G_m\cong H$  by left translations of $G$ has
no fixed point. By Corollary \ref{Cor2.4}, we have $\lambda(G)=0$. If $G$ is unipotent, then $G$ is isomorphic to $K^n$. Hence $\lambda(G)=1$.
\end{proof}

\section{Examples of  additive invariants}

\subsection{Euler characteristic}

When $K$ is a subfield of $\C$, the  Euler characteristic is given by
$$
\chi(X):=\sum_{i}(-1)^i \rank H^i(X(\C),\C).
$$

For more general $K$ and  a variety $X$ over $K$,
let $H^i(X,\Z_l)$ be the $l$-adic cohomology group of $X$, where  $l$ is a positive integer prime to the characteristic $\char(K)$ of $K$.
Set $H^i(X,\Q_l):=H^i(X,\Z_l)\otimes_{\Z_l} \Q_l$. Denote by $\beta^i(X,l):=\dim_{\Q_l}H^i(X,\Q_l)$
the  $i$-th $l$-adic Betti number of $X$. The $l$-adic Euler characteristic is defined by
$$\chi(X,l):=\sum_i (-1)^i\beta^i(X,l).$$

Similarly, let $H_c^i(X,\Z_l)$ be the $l$-adic cohomology group
 of $X$ with compact support. Set $\beta^i_c(X,l):=\dim_{\Q_l}H_c^i(X,\Q_l)$
 the $i$-th $l$-adic Betti number of $X$ with compact support and
$$\chi_c(X,l):=\sum_i (-1)^i\beta_c^i(X,l)$$  the $l$-adic Euler-Poincar\'{e} characteristic
with compact support. Note that $\chi_c(X,l)$ is independent of the
choice of $l$ prime to $\char(K)$ (See, e.g., \cite{Katz} or \cite{Illusie}).

Those $\chi, \chi_c, \chi(-,l)$ and $\chi_c(-,l)$ are  additive invariants from $Var_k$ to $\Z$,
which follows from the fact that $\chi=\chi_c$  and $\chi(-,l)=\chi_c(-,l)$
(cf. \cite{Fulton2} for the case over $\C$, \cite{Laumon} for general cases).

From Corollary \ref{Cor2.4} and note that both $\chi(\G_m)$ (in the case that $k$ is a subfield of $\C$) and
$\chi(\G_m,l)$ are zero. So one gets Bia{\l}ynicki-Birula's result.
\begin{corollary}[\cite{Bialynicki-Birula2}]
 Suppose that $X$ admits a $\G_m$-action with the fixed point set $X^{\G_m}$. Then we have
\begin{enumerate}
\item  $\chi(X)=\chi(X^{\G_m})$ if $K$  is a subfield of $\C$.

\item $\chi(X,l)=\chi(X^{\G_m},l)$ if $\char(K)$ is positive.

\end{enumerate}
\end{corollary}

\subsection{Hodge polynomials}
In this subsection, we assume that $K$ is a field of characteristic zero. Then there is an additive invariant
$H:Var_K\to \Z[u,v]$, with the properties:
\begin{enumerate}
 \item $H_X(u,v):=\sum_{p,q}(-1)^{p+q}\dim H^{q}(X,\Omega_X^p)u^pv^q$ if $X$ is nonsingular and  projective (or complete).
\item  $H_X(u,v)=H_U(u,v)+H_Y(u,v)$ if $Y$ is a closed algebraic subset of $X$ and $U=X-Y$.
\item  If $X=Y\times Z$, then $H_X(u,v)=H_Y(u,v)\cdot H_Z(u,v)$.
\end{enumerate}

The existence and uniqueness of such a polynomial follow from Deligne's Mixed Hodge theory(cf. \cite{Deligne1,Deligne2}).
The coefficient of $u^pv^q$ of $H_X(u,v)$ is called the \emph{virtual Hodge $(p,q)$-number} of $X$ and we denote it by $\tilde{h}^{p,q}(X)$.
Note that from the definition, $\tilde{h}^{p,q}(X)$ coincides with the usual Hodge number $(p,q)$-number ${h}^{p,q}(X)$ if $X$ is a smooth
projective variety.
 To apply the results in the last section, we need suitable modifications. Since $\G_m\cong \Spec(K[x,x^{-1}])$, we have
$H_{\G_m}(u,v)=uv-1\neq 0$. So Corollary \ref{Cor2.4} can not be applied directly to the additive invariant $H$. However, if
we take the values of $H$  in the quotient $\Z[u,v]/\langle uv-1\rangle$, i.e., the composed map of $H$ with the quotient
homomorphism $ \Z[u,v]\to  \Z[u,v]/\langle uv-1\rangle$, then  we get a new additive invariant $\widetilde{H}:Var_K\to
 \Z[u,v]/\langle uv-1\rangle\cong \Z[u,u^{-1}]$. This modified additive invariant $\widetilde{H}$ satisfies  $\widetilde{H}(\G_m)=0$.
 The  following result is from Corollary \ref{Cor2.4}.

\begin{corollary}\label{Cor3.2}
Suppose that $X$ admits a $\G_m$-action with the fixed point set $X^{\G_m}$. Then
$$\widetilde{H}_X(u)=\widetilde{H}_{X^{\G_m}}(u)\in \Z[u,u^{-1}].$$
\end{corollary}

Equivalently, Corollary \ref{Cor3.2} can be written in a different way as the following:
\begin{equation}\label{eq01}
 \sum_{p-q=i}\tilde{h}^{p,q}(X)= \sum_{p-q=i}\tilde{h}^{p,q}(X^{\G_m}).
\end{equation}
In particular, if $X^{\G_m}$ is  dimension zero over $K$, then $\widetilde{H}_{X^{\G_m}}(u)$ is independent of $u$
and so is $\widetilde{H}_{X}(u)$.
In this case, the virtual Hodge $(p,q)$-numbers
of $X$ satisfies the the following relation:
\begin{equation}\label{eq002}
 \sum_{p-q=i}\tilde{h}^{p,q}(X)=0
\end{equation}
for all $i\neq 0$ and
$$
\sum_{p\geq0}\tilde{h}^{p,p}(X)=\chi(X).
$$

More generally, Equation (\ref{eq002}) holds for all $|i|> \dim X^{\G_m}$.
In the case that $X$ is smooth and projective over $k$, then $\tilde{h}^{p,q}(X)={h}^{p,q}(X)\geq 0$. So we obtain
from Equation (\ref{eq01}) that $\sum_{p-q=i}{h}^{p,q}(X)= \sum_{p-q=i}{h}^{p,q}(X^{\G_m})$. In particular,
${h}^{p,q}(X)=\tilde{h}^{p,q}(X)=0$ for all $p,q$ such that $|p-q|>\dim (X^{\G_m})$. This is a pretty simpler proof of a slight weaker version
of Bia{\l}ynicki-Birula decomposition theorem (cf. \cite{Bialynicki-Birula}).

For example, if $X$ is nonsingular toric projective variety over $\C$, then one has $h^{p,q}(X)=0$ for all $p\neq q$ and
$h^{p,p}(X)=\beta^{2p}(X)$. This follows from the fact that a nonsingular projective
toric variety admits a $\G_m$-action with finite isolated fixed points.

\medskip
Now we consider algebraic varieties  admitting actions of the additive group $\G_a$.  Since $\G_a\cong \Spec (K[x])$,
we have $H_{\G_a}(u,v)=uv$. To apply Corollary \ref{Cor2.4}, we need to take the value in
$\Z[u,v]/\langle uv\rangle$. That is, the composed map of $H$ with the quotient map $\Z[u,v]\to \Z[u,v]/\langle uv\rangle$
gives us an additive invariant $\overline{H}_X(u,v)$ such that $\overline{H}_{\G_a}(u,v)=0$.

\begin{corollary}\label{Cor3.4}
Suppose that $X$ admits a $\G_a$-action with the fixed point set $X^{\G_a}$. Then
$$\overline{H}_X(u,v)=\overline{H}_{X^{\G_a}}(u,v)\in \Z[u,v]/\langle uv\rangle.$$
\end{corollary}

Corollary \ref{Cor3.4} implies that the following equations
\begin{equation*}
\tilde{h}^{p,0}(X)= \tilde{h}^{p,0}(X^{\G_m})
\end{equation*}
and
\begin{equation*}
\tilde{h}^{0,q}(X)= \tilde{h}^{0,q}(X^{\G_m})
\end{equation*}
hold for $p,q\geq0$.

In the case that $X$ is smooth and projective with a $\G_a$-action, we have  ${h}^{p,0}(X)= \tilde{h}^{p,0}(X)=0$
and ${h}^{0,q}(X)= \tilde{h}^{0,q}(X)=0$ for $p,q>\dim X^{\G_a}$.

By applying Corollary \ref{Cor3.4} to an algebraic connected affine group variety, we have the following result.
\begin{corollary}
 Let $G$ be an algebraic connected affine group variety. Then
$$
\tilde{h}^{p,0}(G)=\tilde{h}^{0,q}(G)=0
$$
for all $p,q>0$. In particular, the $\rm 1st$ virtual Betti number of $G$ is zero.
\end{corollary}
\begin{proof}
If $G$ is not a torus, then $G$ contain a subgroup $H$ isomorphic to $\G_a$.  The action of $\G_a\cong H$  by left translations of
$G$ has no fixed point. By Corollary \ref{Cor3.4}, we get $\tilde{h}^{p,0}(G)=\tilde{h}^{0,q}(G)=0$
for all $p,q\geq 0$. If $G=\G_m^{\times n}$, then we get  $H_G(u,v)=(uv-1)^n$ since $H$ is an additive invariant and  $H_{\G_m}(u,v)=uv-1$.
From this formula we get immediately that $\tilde{h}^{p,0}(G)=\tilde{h}^{0,q}(G)=0$
for all $p,q>0$.
\end{proof}

\subsection{Counting points}

Let $\F_q$ be the finite field of $q$ elements and let $X$ be an algebraic scheme defined over  $\F_q$.
Let $N_n(X)$ denote the number of closed points in $X(\F_{q^n})$.  Note that $N_n$ defines an additive invariant
on the category $Var_{\F_q}$
of algebraic varieties over the field ${\F_q}$ with integer values,
is a map
$$
N_n:Var_{\F_q}\to \Z
$$
such that
$$\left\{
\begin{array}{lll}
N_n(X)=N_n(X')&\hbox{for $X\cong X'$,}\\
 N_n(X)=N_n(Y)+N_n(X-Y)& \hbox{for $Y$ closed in  $X$,}\\
 N_n(X\times Y)=N_n(X)\cdot N_n(Y)&\hbox{for every $X$ and  $Y$.}
\end{array}\right.
$$

For example, if $X=A_{\F_q}^1=\Spec(\F_q[x])$ the affine line over $\F_q$, then $N_n(X)=q^n$; if $X=\Spec(\F_q[x,x^{-1}])$,
then $N_n(X)=q^n-1$.

\begin{lemma}\label{lemma3.6}
Let $X$, $Y$ and $A$ be algebraic varieties over $\F_q$. Suppose that $X$ is $A$-equivalent to $Y$. Then
\begin{enumerate}
 \item if $A=\Spec(\F_q[x])$ then $N_n(X)\equiv N_n(Y) ~\mod (q)$.
\item if $A=\Spec(\F_q[x,x^{-1}])$ then $N_n(X)\equiv N_n(Y) ~\mod ~(q-1)$.
\end{enumerate}
\end{lemma}
\begin{proof}
It is enough to consider the case that $X$ is simply $A$-equivalent to $Y$ since $N_n$ is an additive invariant.
By definition, there is an open quasi-projective
scheme $U$ of $X$ such that $X-U$ is isomorphic to $Y$ and an isomorphism $f:U\to U'\times A$, where $U'$ is an algebraic scheme.
In this case, we have $N_n(X)=N_n(Y)+N_n(U)=N_n(Y)+N_n(U')N_n(A)$.

Case (1). If $A=\Spec(\F_q[x])$, then $N_n(A)=q^n\equiv 0 ~mod (q)$. Hence   $N_n(X)\equiv N_n(Y)~\mod (q)$.

Case (2). If $A=\Spec(\F_q[x,x^{-1}])$, then $N_n(A)=q^n-1\equiv 0 ~mod (q-1)$. Hence   $N_n(X)\equiv N_n(Y)~\mod (q-1)$.

This completes the proof of the lemma.
\end{proof}

As an application of Theorem \ref{Th2.2} and Lemma \ref{lemma3.6}, we have the following result.

\begin{corollary} Suppose that an algebraic scheme $X$ admits $G$-action.
\begin{enumerate}
\item If $G=\G_a$, then $N_n(X)\equiv N_n(X^G) ~\mod (q)$.
\item If $G=\G_m$, then $N_n(X)\equiv N_n(X^G) ~\mod (q-1)$.
\end{enumerate}

\end{corollary}
\begin{proof}
 This follows from the combination of Theorem \ref{Th2.2} and Lemma \ref{lemma3.6}.
 \end{proof}

\section{Additive invariants for Chow varieties }
\subsection{The Chow variety for projective spaces}
In this section we give a direct proof of Corollary \ref{cor1.5} and Theorem \ref{Th1.4}, which not only  is a
 simplification of Lawson and Yau's proof for Theorem \ref{Th01.1} but also works for  Chow varieties over arbitrary algebraically closed field.

Now we give a proof of Corollary \ref{cor1.5} by using Bia{\l}ynicki-Birula's result.

\begin{proof}[The  proof of Corollary \ref{cor1.5}]
We consider the action of $\G_m$ on $\P^{n+1}_K$ given by setting
$$\Phi_t([z_0,...,z_{n},z_{n+1}])=[z_0,...,z_{n},tz_{n+1}],$$
where $t\in \G_m$ and $[z_0,...,z_{n},z_{n+1}]$ are homogeneous coordinates for $\P^{n+1}_K$.

This action on $\P^{n+1}_K$ induces an action of $\G_m$ on $C_{p+1,d}(\P^{n+1})_K$.
From the definition of the action $\G_m$ on $\P^{n+1}_K$, it is pretty clear that any subvariety $V$ of $\dim V=p+1$
is invariant under the action $\G_m$ if the support of $V$ is included in the hyperplane $(z_{n+1}=0)\cong\P^n_K$.

We also observe that if a $(p+1)$-dimensional irreducible
algebraic variety $V$ is defined by a collection of homogeneous polynomials $F_{\lambda}$ on $\P^{n+1}_K$,
but those polynomials are independent of the last coordinate $z_{n+1}$, then $V$ is invariant under $\G_m$.
Geometrically, such a variety $V$ is a cone of over an algebraic subvariety supported in the hyperplane $(z_{n+1}=0)$.

Denote $Q=[0:\cdots:0:1]\in \P^{n+1}_K$ and note that $Q$ is $\G_m$-fixed.
Note that  only those varieties are  irreducible  invariant subvarieties  of  dimension $p+1$ in $\P^{n+1}_K$
under this $\G_m$-action.   To see this, we first observe from the definition of the action that if an irreducible variety
$V$ contains  $Q$ and another fixed point $P$ on $\P^n_K$, then so does the projective line $l_{PQ}$ passing $P$ and $Q$.
Suppose $V\subseteq \P^{n+1}_K$ such that $V\nsubseteq (z_{n+1}=0)\cong \P^n_K$. Since
both $V$ and $\P^n_K$ are $\G_m$-invariant,  $V':=V\cap \P^n_K$ is $\G_m$-invariant. The subvariety $V$ corresponds to the fixed
point set of the restriction of the $\G_m$-action on $V$ when $t\to 0$.
Therefore, the cone $\Sigma_Q V'$ is $\G_m$-invariant.  Note that  we must have $Q\in V$. The point $Q$ corresponds to the fixed
point set of the restriction of the $\G_m$-action on $V$ when $t\to \infty$.
Hence we have $\Sigma_Q V'\subseteq V$. Since $\dim \Sigma_Q V'=p+1=\dim V$ and $V$ is irreducible, we have $\Sigma_Q V'= V$.

The fixed point set $C_{p+1,d}(\P^{n+1})_K^{\G_m}$ of the induced action on $C_{p+1,d}(\P^{n+1})_K$
contains cycles $c$ of the form
 $c=\sum n_kV_k+\sum m_jW_j$ of degree $\deg c:=\sum n_k\deg V_k+\sum m_j\deg W_j=d$, where $V_k\subset \P^n_K$
is irreducible and $W_j=\Sigma_Q W_j'$ for some irreducible variety $W_j\subset \P^n_K$ of $\dim W_j'=p$. Therefore, we have
\begin{equation}\label{eq02}
C_{p+1,d}(\P^{n+1})_K^{\G_m}=\coprod_{i=0}^d\{ C_{p+1,i}(\P^{n})_K\times\Sigma_Q C_{p,d-i}(\P^{n})_K\}.
\end{equation}
Since $\Sigma:C_{p,d-i}(\P^{n})_K\to C_{p,d-i}(\P^{n+1})_K$ induces a homeomorphism onto its image in $C_{p,d-i}(\P^{n+1})_K$, we have
\begin{equation}\label{eqn2}
\begin{array}{ccl}
 \chi(C_{p+1,d}(\P^{n+1})_K^{\G_m},l)&=&\chi(\coprod_{i=0}^d\{ C_{p+1,i}(\P^{n})_K\times\Sigma_Q C_{p,d-i}(\P^{n})_K\},l)\\
 &=&\sum_{i=0}^d \chi(C_{p+1,i}(\P^{n})_K\times \Sigma_Q C_{p,d-i}(\P^{n})_K,l)\\
&=&\sum_{i=0}^d \chi(C_{p+1,i}(\P^{n})_K,l)\cdot \chi(\Sigma_Q C_{p,d-i}(\P^{n})_K,l)\\
&=&\sum_{i=0}^d \chi(C_{p+1,i}(\P^{n})_K,l)\cdot \chi(C_{p,d-i}(\P^{n})_K,l),
\end{array}
\end{equation}
where the second equality follows from the exclusion-inclusion principle of  the Euler-Poincar\'{e} characteristic (cf. \cite{Laumon},
\cite{Hu}), the third equality follows from the K\"{u}nneth formula for $l$-adic cohomology.

From Theorem \ref{Th01.2}, we have
\begin{equation}\label{eqn3}
 \chi(C_{p+1,d}(\P^{n+1})_K,l)=\chi(C_{p+1,d}(\P^{n+1})_K^{\G_m},l).
\end{equation}

The combination of Equation (\ref{eqn2}) and (\ref{eqn3}) gives us a recursive formula
\begin{equation}\label{eqn4}
 \chi(C_{p+1,d}(\P^{n+1})_K,l)=\sum_{i=0}^d \chi(C_{p+1,i}(\P^{n})_K,l)\cdot \chi(C_{p,d-i}(\P^{n})_K,l).
\end{equation}

The above idea also can be used to calculate the initial values $\chi(C_{0,d}(\P^{n})_K,l)$ as follows.
By definition, an element in $C_{0,d}(\P^{n+1})_K$ means an effective cycle $c$ on $\P^{n+1}_K$ such that $\deg c=d$.
Since a point $P$ is a
fixed point of the $\G_m$-action if and only if $P=Q$ or $P\in (z_{n+1}=0)\cong\P^n_K$, we get $c\in C_{0,d}(\P^{n+1})_K^{\G_m}$
if and only if $c=m Q+\sum n_iP_i$, where $n_i\geq 0$ and $\sum n_i=d-m$. Hence
\begin{equation}\label{eqn06}
C_{0,d}(\P^{n+1})_K^{\G_m}=\coprod_{m=0}^d C_{0,d-m}(\P^{n})_K.
\end{equation}

This together with Theorem \ref{Th01.2} implies the following formula for the Euler-Poincar\'{e} characteristics.
\begin{equation*}
\chi(C_{0,d}(\P^{n+1})_K,l)=\sum_{m=0}^d \chi(C_{0,d-m}(\P^{n})_K,l).
\end{equation*}
From this recursive formula, we get
\begin{equation}\label{eqn5}
 \chi(C_{0,d}(\P^{n})_K,l)=(^{n+d}_{~d}).
\end{equation}

The combination of Equation (\ref{eqn4}) and (\ref{eqn5}) completes the alternate  proof of Corollary \ref{cor1.5}.
\end{proof}

If we set
$$ Q_{p,n}(t):=\sum_{d=0}^{\infty} \chi(C_{p,d}(\P^n)_K,l) t^d,
$$
then Corollary \ref{cor1.5} may be restated as
\begin{equation*}
Q_{p,n}(t)=\bigg(\frac{1}{1-t}\bigg)^{(^{n+1}_{p+1})}, \quad\hbox{where $\chi(C_{p,0}(\P^n)_K):=1$} .
\end{equation*}

\begin{proof}[The proof of Theorem \ref{Th1.4}]
Note that  Bialynicki-Birula's result in \cite[Th.2]{Bialynicki-Birula2}  and the assumption $\lambda(\G_m)=0$ imply that
 if $X$ is a  projective algebrac set over $K$ with a $\G_m$-action, then
$$\lambda(X)=\lambda(X^{\G_m}),$$
 where $X^{\G_m}$ is the fixed point set of this action.

 This together with Equation (\ref{eq02}) gives us the following recursive formula
\begin{equation}\label{eqn12}
 \lambda(C_{p+1,d}(\P^{n+1})_K)=\sum_{i=0}^d \lambda(C_{p+1,i}(\P^{n})_K)\cdot \lambda(C_{p,d-i}(\P^{n})_K).
\end{equation}

By  Equation (\ref{eqn06}) we get
\begin{equation*}
 \lambda(C_{0,d}(\P^{n+1})_K)=\sum_{m=0}^d  \lambda(C_{0,d-m}(\P^{n})_K).
\end{equation*}

This recursive formula together with the assumption $\lambda(\Spec K)=1$ implies that
\begin{equation}\label{eqn13}
  \lambda(C_{0,d}(\P^{n})_K)=(^{n+d}_{~d}).
\end{equation}

By the same argument as in the proof of Corollary \ref{cor1.5}, we obtain the formula in the theorem
from Equation (\ref{eqn12}) and (\ref{eqn13}).

\end{proof}

In the subsection below we will compute the virtual Hodge polynomial and numbers of the Chow varieties for
 projective spaces  over an algebraically closed subfield of $\C$.

\begin{corollary}\label{cor8}
Assume that $\char(K)=0$ and let $\widetilde{H}: Var_K\to \Z[u,u^{-1}]$ be given as above. Then we have
$$
\widetilde{H}(C_{p,d}(\P^n)_K)=\big(^{v_{p,n}+d-1}_{\quad\quad d}\big)\in \Z[u,u^{-1}].
$$
\end{corollary}
\begin{proof}
It is easy to check from the definition of  $\widetilde{H}$ that $\widetilde{H}(\Spec K)=1$ and
 ${H}(\G_m)=uv-1$. So  $\widetilde{H}(\G_m)=0$. Now the corollary follows from Theorem \ref{Th1.4}.
\end{proof}

\begin{remark}
From Corollary \ref{cor8}, $\widetilde{H}(C_{p,d}(\P^n)_K)$ is independent of $u$.  This implies that
the coefficients of both $u$ and $v$ in the Hodge polynomial ${H}(C_{p,d}(\P^n)_K)$ vanish.
\end{remark}

By applying Corollary \ref{Cor3.2} to the Chow variety $C_{m,d}(\P^n_K)$ parameterizing algebraic $m$-cycles
of degree $d$ in the projective space $\P^n_K$, we have the following result.

\begin{corollary}\label{Cor3.3}
For integers $n\geq m\geq 0$ and $d\geq 0$, the  virtual Hodge $(p,q)$-number
of the Chow variety $C_{m,d}(\P^n_K)$ satisfies the following equations:
$$
\sum_{p-q=i}\tilde{h}^{p,q}(C_{m,d}(\P^n_K))=0
$$
for all $i\neq 0$ and
$$
\sum_{p\geq0}\tilde{h}^{p,p}(C_{m,d}(\P^n_K))=\chi(C_{m,d}(\P^n_K)).
$$

\end{corollary}
\begin{proof}
Instead of proving that  there is a $\G_m$-action on $C_{p,d}(\P^n_K)$ such that the fixed point set consists of isolated points,
 we show that there is a ($\G_m)^{\times n}$-action on $C_{p,d}(\P^n_K)$ such that whose fixed point set is finite.
This is shown in the proof of Corollary \ref{cor1.5} by constructing a  sequence of $\G_m$-action on $C_{p,d}(\P^n_K)$
and its fixed points set. The initial idea for such a construction over the complex number field is from Lawson and Yau
(cf. \cite{Lawson-Yau}). Now the equation follows from Equation (\ref{eq03}).
\end{proof}

By applying Corollary \ref{Cor3.4} to the Chow variety $C_{m,d}(\P^n_K)$, we get the following result.

\begin{corollary}\label{Cor3.5}
For integers $n\geq m\geq 0$, $d\geq 0$ and $p,q>0$, the  virtual Hodge $(p,0)$-number and $(0,q)$-number
of the Chow variety $C_{m,d}(\P^n_K)$ vanish.
\end{corollary}
\begin{proof}
Note that  $C_{m,d}(\P^n_K)$ admits an action of a unitriangular group $J$ with exact one fixed point (cf. \cite[\S7]{Horrocks}).
 The general fact is that the group $J$ contains a subgroup isomorphic to $\G_a$. Therefore,
 $\overline{H}_{C_{m,d}(\P^n_K)}(u,v)=1\in \Z[u,v]/\langle uv\rangle$. This implies
that $${H}_{C_{m,d}(\P^n_K)}(u,v)=1+uvg(u,v)\in \Z[u,v],$$
where $g(u,v)\in \Z[u,v]$. Since $\tilde{h}^{p,0}_{C_{m,d}(\P^n_K)}$ is the coefficient of $u^p$ in the Hodge polynomial
 ${H}_{C_{m,d}(\P^n_K)}(u,v)$, we obtain $\tilde{h}^{p,0}_{C_{m,d}(\P^n_K)}=0$ from the above explicit  formula for
${H}_{C_{m,d}(\P^n_K)}(u,v)$. By the same reason, $\tilde{h}^{0,q}_{C_{m,d}(\P^n_K)}=0$.
\end{proof}

Note that for an algebraic variety over $k$,  if we set $\tilde{\beta}^i(X):=\sum_{p+q=i}\tilde{h}^{p,q}(X)$, then we get
the virtual Poincar\'{e} polynomial $\widetilde{P}_X(t)=\sum_i \tilde{\beta}^i(X)t^i$ (cf. \cite[p.92]{Fulton2}) and
$\tilde{\beta}^i(X)$ is called the  $i$th \emph{virtual Betti number} of $X$.

From Corollary \ref{Cor3.5}, we get $\tilde{\beta}^1(C_{m,d}(\P^n_K))=0$.  It can not be obtained by the vanishing of
the usual $\rm 1st$ Betti number of $C_{m,d}(\P^n_{\C})$.

From Corollary \ref{Cor3.3}, we have
\begin{equation}\label{eq03}
 \sum_{i\geq 1} \tilde{\beta}^{2i-1}(C_{m,d}(\P^n_K))=0
\end{equation}
and
\begin{equation*}
\sum_{i\geq 0} \tilde{\beta}^{2i}(C_{m,d}(\P^n_K))=\chi(C_{m,d}(\P^n_K)).
\end{equation*}
In particular, if one could verify that $\tilde{\beta}^{2i-1}(C_{m,d}(\P^n_K))$ are nonnegative for all $i$, then Equation (\ref{eq03})
would imply that the vanishing of all odd virtual Betti numbers.

\begin{remark} Note that the usual $\rm 1st$ Betti number of $C_{m,d}(\P^n_{\C})$ is zero. This is implied by the fact that
$C_{m,d}(\P^n_{\C})$ is simply connected (cf. \cite{Horrocks} or \cite[Lemma 2.6]{Lawson1}).  However, The simply connectedness
of an algebraically closed set $X$ does \textbf{not} imply the vanishing of  the $\rm 1st$ Betti number of $X$. For example,
let $X=C(E)\cup \P^2_{\C}$, where $C(E)$ is a projective cone of a smooth plane cubic $E$ in the
plane $\P^2_{\C}$ and $E=C(E)\cap \P^2_{\C}$. The fact that $\pi_1(X)=0$ follows from a direct application of Van Kampen theorem.
 An elementary calculation yields $\widetilde{H}_X(u,v)=1+u+v+uv-u^2v-uv^2+2u^2v^2$ and so $\tilde{\beta}^1(X)=2\neq 0$.

\end{remark}

The next result, as an application of the above theorem,  we  count  points
of $C_{p,d}(\P^{n})(\F_{q^m})$ modulo $(q-1)$. Recall that the proof of Theorem 2
in \cite{Bialynicki-Birula2} does not require $K$ to be an algebraically closed field. When $K=\F_q$,
the map  $N_m: X\to |X(\F_{q^m})|$ gives rise to an  additive invariant $N_m:Var_K\to \Z$.

\begin{corollary}
 Let $N_m: X\to |X(\F_{q^m})|$ be given as above. For any integer $m\geq1$, we have
$$
N_m(C_{p,d}(\P^n)_{\F_{q^m}}) \equiv\big(^{v_{p,n}+d-1}_{\quad\quad d}\big) ~\mod(q-1).
$$
\begin{proof}
 Since $N_m(\G_m)=q^m-1\equiv 0 ~\mod(q-1)$  and $N_m(\Spec \F_{q})=1$, the corollary follows from Theorem \ref{Th1.4}.
\end{proof}

\end{corollary}

In particular, when $X$ is the Chow variety $C_{m,d}(\P^n_K)$ over $k$, we get
\begin{enumerate}
 \item[(1)] $$N_n(C_{m,d}(\P^n_K))\equiv 1 ~\mod (q).$$
\item[(2)]  $$N_n(C_{m,d}(\P^n_K))\equiv \big(^{v_{p,n}+d-1}_{\quad\quad d}\big) ~\mod (q-1),
 ~where ~v_{p,n}=(^{n+1}_{p+1}).$$
 \end{enumerate}
 \begin{proof}
 (1) follows from  the fact that there is a sequence of $\G_a$-actions on $C_{m,d}(\P^n_K)$ and the last one has exactly
one fixed point (cf. \cite{Horrocks}). (2) follows from the fact that there is a sequence of $\G_m$-actions on
$C_{m,d}(\P^n_K)$ while the last one has exactly $\big(^{v_{p,n}+d-1}_{\quad\quad d}\big)$ isolated fixed points
(cf. the proof of Corollary \ref{cor1.5}).
\end{proof}

For example, the number of points on $k$-point on $C_{m,d}(\P^n_K)$ for $k=\F_2$ is always odd.

\subsection{The Chow variety for the product of  projective spaces}
In this section, we deal with more general cases.
Let $X_K$ be  a projective variety over $K$ (we omit the subscript $K$ below).

Let $\mC_p(X)$ be the topological monoid of all effective $p$-cycles on $X$ and
let $\Pi_p(X)$ be the monoid $\pi_0(\mC_p(X))$ of connected
component of $\mC_p(X)$.  For $\alpha\in \Pi_p(X)$,
let $C_{\alpha}(X)$ be the space of effective algebraic cycles $c$ on $X$
which are in the same connected component $\alpha$.

Under this setting, if we consider the $\G_m$-action on $\P^{n+1}\times X$ by
$$\Phi_t([z_0,...,z_{n},z_{n+1}],x)=([z_0,...,z_{n},tz_{n+1}],x),$$
then for any $\alpha\in \Pi_{p+1}(\P^{n+1}\times X)$, the fixed point set of the
induced $\G_m$ on $C_{\alpha}(\P^{n+1}\times X)$ contains effective cycles of the form
$$
c=\sum n_kV_k+\sum m_jW_j+\sum l_iU_i,  \quad n_k, m_j,l_i\geq 0
$$ whose class in $\Pi_{p+1}(\P^{n+1}\times X)$ is $\alpha$,
where $V_k\subset \P^n\times X$
is irreducible, $W_j=\Sigma_Q W_j'$ for some irreducible variety
$W_j\subset \P^n\times X$ of $\dim W_j'=p$ and $U_k\subset X$ is irreducible of $\dim U_k=p+1$.
Therefore, we have

\begin{equation*}
C_{\alpha}(\P^{n+1}\times X)^{\G_m}=\coprod_{\alpha=\beta+\Sigma_Q\gamma+\gamma' }\{ C_{\beta}(\P^{n}\times X)\times\Sigma_Q
C_{\gamma}(\P^{n}\times X)\times C_{\gamma'}(X)\},
\end{equation*}
where $\beta\in \Pi_{p+1}(\P^{n}\times X)$, $ \gamma\in \Pi_{p}(\P^{n}\times X)$ and $\gamma'\in \Pi_{p+1}( X)$.
Hence we have
{\small
\begin{equation}\label{eqn6}
\chi(C_{\alpha}(\P^{n+1}\times X)^{\G_m},l)=\sum_{\alpha=\beta+\Sigma_Q\gamma+\gamma'}\chi (C_{\beta}(\P^{n}\times X),l)
\cdot\chi(C_{\gamma}(\P^{n}\times X),l)\cdot \chi(C_{\gamma'}(X),l).
\end{equation}
}

By Theorem \ref{Th01.2}, we have
\begin{equation*}
 \chi(C_{\alpha}(\P^{n+1}\times X),l)=\chi(C_{\alpha}(\P^{n+1}\times X)^{\G_m},l)
\end{equation*}

Therefore, by Equation (\ref{eqn6}) we have the following recursive formula
{\small
\begin{equation}\label{eqn7}
\chi(C_{\alpha}(\P^{n+1}\times X) =\sum_{\alpha=\beta+\Sigma_Q\gamma+\gamma'}\chi (C_{\beta}(\P^{n}\times X),l)
\cdot\chi(C_{\gamma}(\P^{n}\times X),l)\cdot \chi(C_{\gamma'}(X),l).
\end{equation}
}

From this we recover the Euler-Poincar\'{e} characteristic of $C_{\alpha}(\P^{n+1}\times X)$
from those of $X$. Therefore, we can obtain the Euler-Poincar\'{e} characteristic for arbitray many products of projective
spaces (cf. \cite{Hu} for a computation without group actions).
\subsection{Toric varieties}
The  result in this subsection is  a formula for  the Euler-Poincar\'{e} characteristic for the Chow variety of
general toric varieties, which is inspired by Elizondo \cite{Elizondo}. For background on toric varieties, the reader is referred
to Fulton's book
\cite{Fulton2}.

Recall that a toric variety over $K$ is an irreducible variety $X$ containing the algebraic
group $T=\G_m^{\times n}$ as a Zariski open
subset such that the action of $\G_m^{\times n}$  on itself extends to an action on $X$.

The \textbf{$p$-th Euler series of $X$} is defined by the following formal power series
$$
E_{p}(X):=\sum_{\alpha\in \Pi_p(X)}\chi(C_{\alpha}(X), l)\alpha.
$$

Since its simplicity, the proof of Theorem \ref{Th5} is given below, which is almost word by word
translated from the case over complex number field (cf. \cite[Th. 2.1]{Elizondo}).

\begin{theorem}\label{Th5}
Denote by $V_1,...,V_N$ the $p$-dimensional invariant irreducible subvarieties of $X$. Let $e_{[V_i]}$ be the
characteristic function  of the subset $\{[V_i], i=1,2,..., N\}$ of $\Pi_p(X)$, where $[V]$ denotes its class
in $\Pi_p(X)$. Then
$$
E_{p}(X)=\prod_{1\leq i\leq N}\bigg(\frac{1}{1-e_{[V_i]}}\bigg).
$$
\end{theorem}
\begin{proof}
 Note first we have $\chi(C_{\alpha}(X),l)=\chi(C_{\alpha}(X)^T,l)$ by applying
Theorem \ref{Th01.2} inductively for $n$-times.  Then $E_p(X)=\prod_{1\leq i\leq N}f_i$, where $f_i(\alpha)=1$ if $\alpha=n\cdot [V_i]$
and $0$ otherwise. Note that $1=(1-e_{[V_i]})\cdot f_i$ since, by definition, $f_i(\alpha)=(1+e_{[V_i]}+e_{[V_i]}^2+\cdots)(\alpha)$
for all $\alpha\in \Pi_p(X)$.
\end{proof}

One needs to know which irreducible subvariety $V$ of a toric variety $X$ is invariant under the action of algebraic torus $T$.
This has been answered in \cite{Elizondo}, i.e., the closure of an orbit under the action $T$. Therefore, theoretically one can
obtain the Euler-Poincar\'{e} characteristic for any toric variety. One may need additional work to get an explicit formula for
$E_p(X)$ in terms of the generators of $\Pi_p$. Elizondo has illuminated how to apply Theorem \ref{Th5} to particular examples
in complex case. His methods also works for the algebraic case. Those examples include projective spaces, the product of two
projective spaces, Hirzebruch surfaces, the blowing up of a projective space a point, etc.

\medskip
Here we give a remark on the Euler-Chow series of certain projective bundles. Let $E_1$ and $E_2$ be two algebraic vector bundle
over a projective variety $X$. Let $\P(E_1)$ (resp. $\P(E_2)$) be the projectivization of $E_1$ (resp. $E_2$). Then,
in complex case, the Euler-Chow series $E_p(\P(E_1\oplus E_2))$ can be computed in terms of that of $\P(E_1)$, $\P(E_2)$ and
$\P(E_1)\times_X \P(E_2)$, where the last one is the fiber product of $\P(E_1)$ and $\P(E_2)$ over $X$ (cf. \cite{Elizondo-Lima}).
The  proof there word for word  works for the algebraic analogue, except that the fixed point formula there is replaced by Theorem \ref{Th01.2}.
As an application, one can obtain the Euler-Chow series for Grassmannians and Flag varieties over $K$.

The calculation of Euler-Poincar\'{e} characteristic for product of projective spaces, or more generally, of toric varieties
 $X$ works well for an additive invariant $\lambda:Var_K\to R$ satisfying $\lambda(\G_m)=0$ and $\lambda(\Spec K)=1$.
For such a $\lambda$, the \textbf{$p$-th  $\lambda$-series of $X$} is defined to be the following formal power series
$$
\Lambda_{p}(X):=\sum_{\alpha\in \Pi_p(X)} \lambda(C_{\alpha}(X))\alpha.
$$

The same formula in Theorem \ref{Th5} holds for $\Lambda_{p}$  on any toric variety $X$.
\begin{theorem} For a toric variety $X$ in Theorem \ref{Th5}, we have
 $$
 \Lambda_{p}(X)=\prod_{1\leq i\leq N}\bigg(\frac{1}{1-e_{[V_i]}}\bigg).
$$
\end{theorem}

\subsection{Chow varieties parameterizing irreducible varieties}
In this subsection, we compute the $l$-adic Euler-Poincar\'{e} characteristic of the Chow varieties parameterizing irreducible
subvarieties of a given dimension and degree in projective spaces.
Let $I_{p,d}(\P^n)_K\subset C_{p,d}(\P^n)_K$ be the subset contains $p$-dimensional
subvarieties of $\P^n_K$ of degree $d$, i.e.,
$$
I_{p,d}(\P^n)_K=\{V\in C_{p,d}(\P^n)_K| \hbox{$V$ is irreducible of  $\deg V=d$ }\}.
$$

Note that  $I_{p,1}(\P^n)_K$ is the Grassmannian of $p+1$-plane in $\P^n_K$, i.e., $I_{p,1}(\P^n)_K=G(p+1,n+1)$.
For $d>1$ each $I_{p,d}(\P^n)_K$ is a finite union of  quasi-projective varieties. The following result is about
the $l$-adic Euler-Poincar\'{e} characteristic of $I_{p,d}(\P^n)_K$.

\begin{theorem}[\cite{Hu}]\label{Th4.9} For $(l,\char(K))=1$, we have
$$\chi(I_{p,d}(\P^n)_K,l)=\left\{
\begin{array}{ccc}
(^{n+1}_{p+1})& \hbox{for $d=1$,}\\
0 &\hbox{for $d> 1$.}
\end{array}
\right.
$$
\end{theorem}

\begin{proof} The proof here is similar to the case over the complex number field.
Recall that the action of the algebraic $n$-torus $T^n:=\G_m^{n+1}/\G_m$ is given by
$$
\Phi_t([z_0,z_1,...,z_n])=[t_0z_0,t_1z_1,...,t_nz_n]
$$
where $t=(t_0,...,t_n)$ and $[z_0,z_1,...,z_n]$ are homogeneous coordinate for $\P^n_K$. This action on
$\P^n$ induces an action of $T$ on $I_{p,d}(\P^n)_K$ and hence on its closure $\bar{I}_{p,d}(\P^n)_K$ in $C_{p,d}(\P^n)_K$ and
$ \bar{I}_{p,d}(\P^n)-{I}_{p,d}(\P^n)$.
By Theorem \ref{Th01.2}, we have
\begin{equation}\label{eq8}
\chi(F_{p,d}(\P^n)_K,l)=\chi(\bar{I}_{p,d}(\P^n)_K).
\end{equation}

By induction on Theorem \ref{Th01.2}, we obtain that a $p$-dimensional $T$-invariant cycles is
a linear combination of $p$-planes,
we get $F_{p,d}(\P^n)_K\subset \bar{I}_{p,d}(\P^n)_K-{I}_{p,d}(\P^n)_K$ for $d>1$,
where $F_{p,d}(\P^n)_K$ is the fixed point set of $T$-action on $\bar{I}_{p,d}(\P^n)_K$.
This together with Theorem \ref{Th01.2} implies that
\begin{equation}\label{eq9}
\chi(F_{p,d}(\P^n)_K,l)=\chi(\bar{I}_{p,d}(\P^n)_K-{I}_{p,d}(\P^n)_K,l).
\end{equation}

By the inclusion-exclusion property for $l$-adic Euler-Poincar\'{e} characteristic (cf. \cite{Laumon}), we have
$\chi(\bar{I}_{p,d}(\P^n)_K-{I}_{p,d}(\P^n)_K,l)=\chi(\bar{I}_{p,d}(\P^n)_K,l)-\chi({I}_{p,d}(\P^n)_K,l)$.
This together with Equation (\ref{eq8}) and (\ref{eq9}) implies ${I}_{p,d}(\P^n)_K,l)=0$ for $d>1$. This case that $d=1$ follows from the fact
$\chi(I_{p,1}(\P^n),l)=\chi(G(p+1,n+1),l)=(^{n+1}_{p+1})$.
\end{proof}

Similarly, we have an alternative shorter proof of  the following result.
\begin{proposition}\label{prop4.10} For $(l,\char(K))=1$, we have
$$\chi(I_{\alpha}(\P^n\times\P^m)_K,l)=\left\{
\begin{array}{ccl}
&(^{n+1}_{k+1})(^{m+1}_{l+1}),& \hbox{if $\alpha=[\P^k_K\times\P^l_K]$, where  $k+l=p$,}\\
&0 ,&\hbox{otherwise.}
\end{array}
\right.
$$
\end{proposition}

\begin{proof}
The action of the algebraic torus $T:=T^n\times T^m$ on $\P^n_K\times\P^m_K$ is given as the product of the actions on each
factor defined in the above proof of Theorem \ref{Th4.9}.
For each $\alpha\in \Ch_{p}(\P^n_K\times\P^m_K)$, where $\Ch_p(X)$ denotes the Chow group of $p$-cycles on $X$,
 the action of $T$ on $\P^n_K\times\P^m_K$ induces an action on $C_{\alpha}(\P^n\times\P^m)_K$ and
$I_{\alpha}(\P^n\times\P^m)_K$ since the rational equivalent class of an irreducible variety is preserved by this action.
Since the action is algebraic, it extends to the closure  $\bar{I}_{\alpha}(\P^n\times\P^m)_K$
of $I_{\alpha}(\P^n\times\P^m)_K$ in $C_{\alpha}(\P^n\times\P^m)_K$.

By Theorem \ref{Th01.2}, we have
$$\chi(F_{\alpha}(\P^n\times\P^m)_K,l)=\chi(\bar{I}_{\alpha}(\P^n\times\P^m)_K,l),$$
where $F_{\alpha}(\P^n\times\P^m)_K$ the fixed point set of this action in $\bar{I}_{\alpha}(\P^n\times\P^m)_K$.

The $T$-invariant cycles in $\alpha$ are exactly finite sum of products of $k$-planes in $\P^n_K$ and $(p-k)$-planes
in $\P^m_K$, where $0\leq k\leq p$. Hence if $\alpha\neq e_{k,l}$ for all $k+l=p,k,l\geq 0$, then
$F_{\alpha}(\P^n\times\P^m)_K\subset \bar{I}_{\alpha}(\P^n\times\P^m))_K-{I}_{\alpha}(\P^n\times\P^m)_K)$. Applying Theorem \ref{Th01.2} to
$\bar{I}_{\alpha}(\P^n\times\P^m))_K-{I}_{\alpha}(\P^n\times\P^m)_K)$, we have
$$
\chi(F_{\alpha}(\P^n\times\P^m)_K,l)=\chi(\bar{I}_{\alpha}(\P^n\times\P^m)_K-{I}_{\alpha}(\P^n\times\P^m)_K,l).
$$

These two equations together the inclusion-exclusion property
for $l$-adic Euler-Poincar\'{e} characteristic  imply that
$$
\chi({I}_{\alpha}(\P^n\times\P^m)_K,l)=0.
$$

If $\alpha=e_{k,l}$ for some $k,l\geq 0, k+l=p$, then $I_{\alpha}(\P^n\times \P^m)_K=G(k+1,n+1)\times G(l+1,m+1)$ and so
$\chi(I_{\alpha}(\P^n\times \P^m)_K,l)=(^{n+1}_{k+1})(^{m+1}_{l+1})$.

The completes the proof of Proposition \ref{prop4.10}.
\end{proof}

From the proof of Theorem \ref{Th4.9} and Proposition \ref{prop4.10}, we observe that it works nicely for general additive invariants
$\lambda:Var_K\to R$ satisfying $\lambda(\G_m)=0$ and $\lambda(\Spec K)=1$.
That is,  the following statement holds.

\begin{proposition}\label{prop12} For  additive invariants  $\lambda:Var_K\to R$
satisfying $\lambda(\G_m)=0$ and $\lambda(\Spec K)=1$,
we have
$$\lambda(I_{p,d}(\P^n)_K)=\left\{
\begin{array}{ccc}
(^{n+1}_{p+1})& \hbox{for $d=1$,}\\
0 &\hbox{for $d> 1$.}
\end{array}
\right.
$$
and
$$ \lambda(I_{\alpha}(\P^n\times\P^m)_K)=\left\{
\begin{array}{ccl}
&(^{n+1}_{k+1})(^{m+1}_{l+1}),& \hbox{if $\alpha=[\P^k_K\times\P^l_K]$, where  $k+l=p$,}\\
&0 ,&\hbox{otherwise.}
\end{array}
\right.
$$
\end{proposition}

This proposition has the following immediate corollary.
\begin{corollary}
The Hodge polynomial $H(I_{p,d}(\P^n)_K)\in \Z[u,v]$ (resp. $H(I_{\alpha}(\P^n\times\P^m)_K)$) is in the ideal
 $\langle uv-1\rangle$
generated by $uv-1$ for $d>1$(resp. $\alpha\neq [\P^k_K\times\P^l_K]$).
\end{corollary}

\end{document}